\documentclass[fleqn, 12pt]{article}
\pdfoutput=1
\usepackage[ngerman,american]{babel}
\usepackage[utf8]{inputenc}
\usepackage{geometry}
\usepackage{array}
\usepackage{float}
\usepackage[urlcolor=blue]{hyperref}
\usepackage{dsfont}
\usepackage{subcaption}
\usepackage{latexsym}
\usepackage{graphicx}
\usepackage{mathtools, amssymb,amsthm}
\usepackage[shortlabels]{enumitem}
\usepackage{pdfpages}
\usepackage[]{algorithm2e}
\usepackage{pdflscape}
\usepackage{xspace}
\usepackage{setspace}
\usepackage{xcolor}
\usepackage{diagbox}
\usepackage{booktabs}

\usepackage[authoryear,round,colon,sort&compress]{natbib}
\bibpunct{(}{)}{,}{autheryear}{,}{,}

\definecolor{airforceblue}{rgb}{0.36, 0.54, 0.66}
\definecolor{darkergreen}{rgb}{0.0, 0.5, 0.0}
\hypersetup{
    colorlinks,
    linkcolor={darkergreen},
    citecolor={airforceblue},
    urlcolor={blue!80!black}
}

\title{Dimension-invariant uniform consistency of the empirical spatial distribution function and its associated spatial depth estimator}

\date{}
\author{Felix Gnettner\footnote{Department of Mathematics and Statistics, South Dakota State University. E-mail: felix.gnettner@sdstate.edu}, \ Hyemin Yeon\footnote{Department of Mathematical Sciences, Kent State University. E-mail: hyeon1@kent.edu}, \ Piotr Kokoszka\footnote{Department of Statistics, Colorado State University. E-mail: piotr.kokoszka@colostate.edu}}
\makeatletter

\newcommand{\E}{\mathbb{E}}
\renewcommand{\P}{\mathbb{P}}

\newcommand{\dint}{\mathrm{d}}
\newcommand{\deq}{\overset{\mathcal D}{=}}

\makeatother

\theoremstyle{plain}
\newtheorem{Theorem}{Theorem}[section]

\newtheorem{Remark}[Theorem]{Remark}
\newtheorem{Lemma}[Theorem]{Lemma}

\expandafter\let\expandafter\oldproof\csname\string\proof\endcsname
\let\oldendproof\endproof
\renewenvironment{proof}[1][\proofname]{\oldproof[\bfseries #1]}{\oldendproof}
\begin{document}
\maketitle
\abstract{We provide a proof that the empirical spatial distribution estimator in $\mathbb R^d$ as well as the corresponding plug-in estimator of the spatial depth are uniformly $L^1$-consistent. The consistency rate only depends on the sample size $n$, not on the dimension $d$ or any tuning or regularization parameters. This is a rare property. The result of this note originates from a conversation with ChatGPT 5.4 Pro as part of some of our own earlier experiments on its mathematical reasoning capabilities.}

\section{Introduction}
Consider iid $\mathbb R^d$-valued random elements $X_1,...,X_n \sim P^X$. The spatial sign of a point $x \in \mathbb R^d$ is defined as $\mathcal S_x = x/\|x \|_{\mathbb R^d}$ (using the convention that $\mathcal S_{0_d} = 0_d$), where $\| \cdot \|_{\mathbb R^d}$ denotes the Euclidean norm. In the following, $\langle \cdot,\cdot,\rangle_{\mathbb R^d}$ denotes the standard Euclidean inner product.

The population spatial distribution function of $P^X$ maps any point $x \in \mathbb R^d$ to the unit ball in this space. It is defined as \[S_x = \int \mathcal{S}_{x-w} \ \dint P^X(w) = \E_X \left( \mathcal S_{x-X} \right), \qquad x \in \mathbb R^d.\] Its empirical version with respect to the sample $X_1,\ldots,X_n$ and its empirical probability measure $\widehat P_n^X$ is \[\widehat S_x = \int \mathcal{S}_{x-w} \ \dint \widehat P_n^X(w) = \frac 1n \sum_{i=1}^n \mathcal S_{x-X_i}, \qquad x \in \mathbb R^d.\]
The corresponding population spatial depth, respectively its empirical version, are defined as \[D_S(x,P^X) = 1-\|S_x\|_{\mathbb R^d}\qquad \text{and} \qquad D_S(x,\widehat P_n^X) = 1-\|\widehat S_x\|_{\mathbb R^d},\]
and the above definitions of $S_x, \widehat S_x$ and $D_S$ can be extended to general separable Hilbert spaces $\mathbb H$ via replacing $\|\cdot \|_{\mathbb R^d}$ by $\| \cdot \|_{\mathbb H}$. 

The spatial distribution function was initially defined by \cite{Chaudhuri96}. \citet[Theorem 2.5]{Koltchinskii97} showed that it uniquely characterizes the underlying probability measure $P^X$ in $\mathbb{R}^d$. It is a foundation for defining spatial ranks and quantiles. The spatial depth \citep{VardiZhang2000,Serfling2002,Gao2003} measures the relative deepness of a point $x \in \mathbb R^d$ with respect to $P^X$. Its argmax defines the spatial or geometric median. If $x$ is a central point, the norm of $S_x$ is close to 0. In case $x$ is outlying, the norm of $S_x$ is closer to 1.  Very recently, \cite{Konen26} proved that the spatial depth also uniquely characterizes the underlying probability measure $P^X$. It is one of the few, if not the only, tuning parameter free depth functions that does not degenerate for $d \to \infty$ \citep{Chakraborty14b}, has a direct analogue in spaces of infinite dimension \citep{Chakraborty14a}, is numerically easy to evaluate, and has multiple generalizations to non-standard data \citep{Konen23}. 

\paragraph*{Motivation and contribution}
There exist only very few constructions of depth functions which have estimators that are uniformly consistent with a dimension-invariant rate. The $h$-depth \citep{Cuevas2007} is probably the most prominent example with this property \citep{Wynne25}. Its sample estimator is basically a classical kernel density estimator with a fixed user-specified bandwidth independent of the sample size $n$ and the dimension $d$. 

For the spatial depth in separable Hilbert spaces $\mathbb H$ of infinite dimension\linebreak \cite{Chakraborty14a} proved, under some technical assumptions, uniform consistency on compact subsets $K\subseteq \mathbb H$, i.e. \[ \sup_{x \in K} \left|D_S(x,\widehat P_n^X)-D_S(x,P^X) \right|  \leq \sup_{x \in K} \left\|\widehat S_x - S_x \right\|_{\mathbb H} = O_P\left( \frac{1}{\sqrt n}\right).\]
Since the unit ball in an infinite-dimensional Hilbert space $\mathbb H$ is not compact, such an assumption can cause difficulties for statistical inference. Similarly, \cite{Yeon25} propose a regularized version of the Tukey depth \citep{Tukey} in separable Hilbert spaces of infinite dimension and prove its uniform consistency on totally bounded subsets.

In the following, we present a short proof that the empirical spatial distribution estimator in $\mathbb R^d$ as well as the corresponding plug-in estimator of the spatial depth are uniformly $L^1$-consistent on the entire space $\mathbb R^d$ with a rate that only depends on the sample size $n$, not on the dimension $d$.
The proof is based on obtaining an upper bound on the Rademacher complexity of the corresponding function class. The Rademacher complexity quantifies the  generalization quality of function estimators. Obtaining dimension-invariant Rademacher bounds plays an important role in statistical learning theory and may help to improve machine learning methods. Statistically meaningful function classes with dimension-invariant Rademacher complexities are rare. Examples include bounded balls in reproducing kernel Hilbert spaces \citep[Lemma 22]{Bartlett2002}, Euclidean norm-bounded linear classes \citep[Theorem 3]{Kakade2008}, certain norm-controlled neural networks \citep{Golowich20,Sellke24}, some bounded classes of finite cardinality \citep[Lemma 5.2]{Massart2000} as well as certain transformations \citep[Theorem 12]{Bartlett2002} of the aforementioned classes.  

The spatial distribution function and the spatial depth do not require any boundedness and moment assumptions, but uniquely characterize the underlying probability measure. This indicates that self normalization, on which our considered function class heavily relies, can be a powerful tool to construct estimators that do not suffer from the curse of dimensionality. The proof incorporates an uncommon technique that does neither rely on covering or packing number nor on compactness arguments, which can cause problems for obtaining dimension-invariant rates. It also does not rely on smoothness assumptions.

\section{Results}
In the following we present the results and the corresponding proofs.

\begin{Theorem}\label{thm_main}
For a sample of iid random elements $X_1,\ldots,X_n \sim P^X$ with values in $\mathbb R^d$, it holds
\[\E \left( \sup_{x \in \mathbb R^d} \left|D_S(x,\widehat P_n^X)-D_S(x,P^X) \right| \right) \leq \E \left(\sup_{x \in \mathbb R^d} \left\|\widehat S_x - S_x \right\|_{\mathbb R^d} \right) \leq \frac{4 \cdot \sqrt{\pi}}{\sqrt n}.\]
\end{Theorem}

\begin{proof}
Consider a class $\mathbb{F}$ of mappings with image in a separable Banach space $B$. The Rademacher complexity of the function class $\mathbb{F}$ with respect to $X_1,...,X_n$ is defined as
\begin{align*}
R_n(\mathbb{F}) = \E\left(\sup_{f \in \mathbb{F}} \left| \left| \frac{1}{n} \sum_{i =1}^n \epsilon_i \cdot f(X_i) \right| \right|_{B}\right).
\end{align*}
Here, $\epsilon_1,...,\epsilon_n$ are independent Rademacher-distributed random variables that do not depend on $X_1,\ldots,X_n$.
We set $B=\mathbb R^d$ and consider the function class $\mathbb{F} = \{f_x: x \in \mathbb R^d\}$, where $f_x(X) = \mathcal S_{x-X}$. Each function $f_x$ in this class is indexed, and thus represented, by some $x \in \mathbb R^d$. This function class is bounded from above by 1. A standard symmetrization argument \citep[Equations (4.17) and (4.18)]{Wainwright2019} entails that \[\E \left(\sup_{x \in \mathbb R^d} \left\|\widehat S_x - S_x \right\|_{\mathbb R^d} \right) \leq 2 \cdot R_n(\mathbb{F}).\] Lemma~\ref{lemma2} below shows that $R_n(\mathbb{F}) \leq 2\sqrt{\pi}/\sqrt n$. So, the assertion follows.

\end{proof}
\begin{Lemma}\label{lemma2}
Consider two independent samples of iid random elements: $X_1,...,X_n$ has values in $\mathbb R^d$ and $\epsilon_1,...,\epsilon_n$ follow the Rademacher distribution. Then,  for any $d \in \mathbb N$ it holds that
\[\E\left(\sup_{x \in \mathbb R^d}\left\|\frac 1n \sum_{i=1}^n \epsilon_i \cdot \mathcal S_{x-X_i}\right\|_{\mathbb R^d} \right) \leq \frac{2\sqrt{\pi}}{\sqrt n}. \]
\end{Lemma}
\begin{proof}
It suffices to prove that \[\E_{\epsilon_1,...,\epsilon_n}\left(\sup_{x \in \mathbb R^d} \left\| \sum_{i=1}^n \epsilon_i \cdot \mathcal S_{x-X_i}\right\|_{\mathbb R^d} \right) \leq 2\sqrt{\pi n} \text{ almost surely}. \]
Then, the assertion follows from the law of irerated expectations.

Jensen's inequality entails that
\begin{equation}\label{eq:1.3}
  \E_{\epsilon_1,...,\epsilon_n}\left(\sup_{x \in \mathbb R^d} \left\| \sum_{i=1}^n \epsilon_i \cdot \mathcal S_{x-X_i}\right\|_{\mathbb R^d} \right) \leq \sqrt{\E_{\epsilon_1,...,\epsilon_n}\left(\sup_{x \in \mathbb R^d} \left\| \sum_{i=1}^n \epsilon_i \cdot \mathcal S_{x-X_i}\right\|_{\mathbb R^d}^2 \right)}.
\end{equation}
So, it suffices to prove that $\E_{\epsilon_1,...,\epsilon_n}\left(\sup_{x \in \mathbb R^d} \left\| \sum_{i=1}^n \epsilon_i \cdot \mathcal S_{x-X_i}\right\|_{\mathbb R^d}^2 \right) \leq 4\pi n$ almost surely.

Let $Y \sim \mathrm N(0_d,I_d)$ be independent of $X_1,...,X_n$ and $\epsilon_1,...,\epsilon_n$. By Lemma~\ref{lemma1} it holds
\begin{align}\label{eq:1}
    \left\| \sum_{i=1}^n \epsilon_i \cdot \mathcal S_{x-X_i}\right\|_{\mathbb R^d}^2 &= \sum_{i=1}^n \sum_{j=1}^n \epsilon_i \epsilon_j \langle \mathcal S_{x-X_i}, \mathcal S_{x-X_j} \rangle_{\mathbb R^d}\nonumber\\
    &\leq  \frac{\pi}{2} \mathbb E_Y \left( \left(\sum_{i=1}^n \epsilon_i \cdot \mathrm{sgn}(\langle Y,\mathcal S_{x-X_i} \rangle_{\mathbb R^d})  \right)^2 \right) \nonumber\\
    &= \frac{\pi}{2} \mathbb E_Y \left( \left(\sum_{i=1}^n \epsilon_i \cdot \mathrm{sgn}(\langle Y,x \rangle_{\mathbb R^d} - \langle Y,X_i \rangle_{\mathbb R^d} )  \right)^2 \right).
\end{align}
The last equality arises from the fact that the norm in the $\mathcal S_{x-X_i}$ does not change the sign, i.e.\ $\mathrm{sgn}(\langle Y,\mathcal S_{x-X_i} \rangle_{\mathbb R^d}) = \mathrm{sgn}(\langle Y,x-X_i \rangle_{\mathbb R^d})$.

Now, \eqref{eq:1} implies that
\begin{multline}\label{eq:1.4}
  \mathbb E_{\epsilon_1,...\epsilon_n}\left(\sup_{x \in \mathbb R^d}\left\| \sum_{i=1}^n \epsilon_i \cdot \mathcal S_{x-X_i}\right\|_{\mathbb R^d}^2\right)\\
\leq \frac{\pi}{2} \mathbb E_{Y}\left( \E_{\epsilon_1,...\epsilon_n} \left( \sup_{x \in \mathbb R^d}\left|\sum_{i=1}^n \epsilon_i \cdot \mathrm{sgn}(\langle Y,x \rangle_{\mathbb R^d} - \langle Y,X_i \rangle_{\mathbb R^d} )  \right|^2 \right)\right).
\end{multline}
Set $Z_i = \langle Y,X_i \rangle_{\mathbb R^d}$. Since $\langle Y,x \rangle_{\mathbb R^d}$ can be any real number, the supremum over $x \in \mathbb R^d$ can be replaced by a supremum over $t \in \mathbb R$, and it holds
\begin{equation}\label{eq:1.5}
  \sup_{x \in \mathbb R^d}\left|\sum_{i=1}^n \epsilon_i \cdot \mathrm{sgn}(\langle Y,x \rangle_{\mathbb R^d} - \langle Y,X_i \rangle_{\mathbb R^d} )  \right| \overset{\text{a.s.}}{=} \sup_{t \in \mathbb R}\left|\sum_{i=1}^n \epsilon_i \cdot \mathrm{sgn}(t - Z_i )  \right|.
\end{equation}
Now, we sort the $Z_i$-values in increasing order, i.e.\ $Z_{\Pi(1)} \leq ... \leq Z_{\Pi(n)}$ and define $\eta_k = \epsilon_{\Pi(k)}$ and $S_k = \sum_{j=1}^k \eta_j$ with $S_0=0$. Since $\Pi$ depends only on $Y,X_1,...,X_n$, which are independent of $\epsilon_1,...,\epsilon_n$, the distribution of $\eta_1,...,\eta_n$ is the same as the distribution of $\epsilon_1,...,\epsilon_n$, i.e.\ $\eta_1,...,\eta_n$ are iid Rademacher distributed random variables. If $t \in (Z_{\Pi(k)},Z_{\Pi(k+1)})$, then \[\mathrm{sgn}(t-Z_{\Pi(j)}) = \begin{cases}
  +1, & j \leq k,\\
  -1, & j > k.
\end{cases} \]
Thus, \[\sum_{i=1}^n \epsilon_i \cdot \mathrm{sgn}(t - Z_i ) = \sum_{j=1}^k \eta_j - \sum_{j=k+1}^n \eta_j = 2S_k-S_n.\]
If the random variables $Z_1,...,Z_n$ take distinct values, it holds \[\sup_{t \in \mathbb R}\left|\sum_{i=1}^n \epsilon_i \cdot \mathrm{sgn}(t - Z_i )  \right| = \max_{k\in \{0,...,n\}} |2S_k-S_n|.\]
If there are ties, it still holds
\begin{equation}\label{eq:2}
  \sup_{t \in \mathbb R}\left|\sum_{i=1}^n \epsilon_i \cdot \mathrm{sgn}(t - Z_i )  \right| \leq \max_{k\in \{0,...,n\}} |2S_k-S_n|.
\end{equation}
Since $\max_{k\in \{0,...,n\}} |2S_k-S_n| \leq \max_{k\in \{0,...,n\}} |S_k| + \max_{k\in \{0,...,n\}} |S_n-S_k|$ and\linebreak $\max_{k\in \{0,...,n\}} |S_k| \deq \max_{k\in \{0,...,n\}} |S_n-S_k|$, it holds \[\E_{\epsilon_1,...\epsilon_n}  \left(\max_{k\in \{0,...,n\}} |2S_k-S_n|^2 \right) \leq 4\E_{\epsilon_1,...\epsilon_n}  \left( \max_{k\in \{0,...,n\}} |S_k|^2 \right).\]
Lévy's inequality entails that $\P_{\epsilon_1,...,\epsilon_n}(\max_{k\in \{0,...,n\}} |S_k| \geq t) \leq 2 \cdot \P_{\epsilon_1,...,\epsilon_n} \left(|S_n| \geq t \right)$ holds for all $t>0$. The layer cake representation of moments yields that,
\begin{align}\label{eq:3}
    \E_{\epsilon_1,...,\epsilon_n}\left(\max_{k\in \{0,...,n\}} |S_k|^2 \right) &= \int_0^\infty 2t \cdot \P_{\epsilon_1,...,\epsilon_n} \left(\max_{k\in \{0,...,n\}} |S_k|>t\right) \ \dint t \nonumber \\
    &\leq 2 \cdot \int_0^\infty 2t \cdot \P_{\epsilon_1,...,\epsilon_n} \left( |S_n|>t\right) \ \dint t \nonumber \\
    &= 2 \cdot \E_{\epsilon_1,...,\epsilon_n} \left( S_n^2 \right).
\end{align}
Moreover,
\begin{equation}\label{eq:4}
  \E_{\epsilon_1,...,\epsilon_n} \left( S_n^2 \right) = n,
\end{equation}
as the summands of $S_n$ are independent Rademacher random variables.

Combining \eqref{eq:4} with \eqref{eq:3}, \eqref{eq:2}, \eqref{eq:1.5}, \eqref{eq:1.4} and \eqref{eq:1.3} proves the assertion.
\end{proof}
\begin{Lemma}\label{lemma1}
  For any vectors $v_1,...,v_n \in \mathbb S^{d-1} \cup \{0_d\}$, any scalars $a_1,...,a_n$ and $Y \sim \mathrm N(0_d,I_d)$, it holds
  \[\sum_{i=1}^n\sum_{j=1}^n a_ia_j \langle v_i,v_j \rangle_{\mathbb R^d} \leq  \frac{\pi}{2} \mathbb E_Y \left( \left(\sum_{i=1}^n a_i \cdot \mathrm{sgn}(\langle Y,v_i \rangle_{\mathbb R^d})  \right)^2 \right).\]
\end{Lemma}
\begin{proof}
  Consider $v_1,...,v_n \in \mathbb S^{d-1}$. The vector $(\langle Y,v_i \rangle_{\mathbb R^d}, \langle Y,v_j \rangle_{\mathbb R^d})^\top$ has a bivariate normal distribution, and the correlation between its entries is $\langle v_i,v_j \rangle_{\mathbb R^d}$. The Sheppard formula \citep[Corollary 3.1]{Li09} entails that \[ \mathbb E_Y \left(\mathrm{sgn}(\langle Y,v_i \rangle_{\mathbb R^d} ) \cdot \mathrm{sgn}(\langle Y,v_j \rangle_{\mathbb R^d}) \right) = \frac{2}{\pi} \arcsin(\langle v_i,v_j \rangle_{\mathbb R^d}).\]
  So,
  \begin{align*}
      &\frac{\pi}{2}\mathbb E_Y \left( \left(\sum_{i=1}^n a_i \cdot \mathrm{sgn}(\langle Y,v_i \rangle_{\mathbb R^d})  \right)^2 \right) \\
&= \frac{\pi}{2} \sum_{i=1}^n \sum_{j=1}^n a_i\, a_j\, \mathbb E_Y \left(\mathrm{sgn}(\langle Y,v_i \rangle_{\mathbb R^d} ) \cdot \mathrm{sgn}(\langle Y,v_j \rangle_{\mathbb R^d} )\right)\\ &= \sum_{i=1}^n \sum_{j=1}^n a_i\, a_j\, \arcsin(\langle v_i,v_j \rangle_{\mathbb R^d}).
  \end{align*}
  Since $\arcsin(t)$ has a power series representation $ \arcsin(t) = t+ \sum_{m=1}^\infty c_m t^{2m+1}$, $t \in [-1,1]$, $c_m>0$, we obtain
  \begin{multline*}
  \sum_{i=1}^n \sum_{j=1}^n a_i\, a_j\, \arcsin(\langle v_i,v_j \rangle_{\mathbb R^d})\\ = \sum_{i=1}^n \sum_{j=1}^n a_i\, a_j\, \langle v_i,v_j \rangle_{\mathbb R^d}+ \sum_{m=1}^\infty c_m \sum_{i=1}^n \sum_{j=1}^n a_i\, a_j\,  \langle v_i,v_j \rangle_{\mathbb R^d}^{2m+1} .
  \end{multline*}
  As $c_m>0$ and the linear kernel is positive semi-definite and products of positive semi-definite kernels remain positive semi-definite kernels, it holds
  \begin{equation*}
  \begin{split}
  \frac{\pi}{2}\mathbb E_Y \left( \left(\sum_{i=1}^n a_i \cdot \mathrm{sgn}(\langle Y,v_i \rangle_{\mathbb R^d})  \right)^2 \right)  &= \sum_{i=1}^n \sum_{j=1}^n a_i\, a_j\, \arcsin(\langle v_i,v_j \rangle_{\mathbb R^d})\\ &\geq \sum_{i=1}^n \sum_{j=1}^n a_i\, a_j\, \langle v_i,v_j \rangle_{\mathbb R^d}.
  \end{split}
  \end{equation*}
  If some of the $v_i$-values are zero, we can simply remove the corresponding summands from the sums, and the assertion still holds.
\end{proof}

\begin{Remark}
Lemma~\ref{lemma2} and Lemma~\ref{lemma1} are based on a lengthy response of ChatGPT 5.4 Pro that contained small gaps and minor flaws. A direct request to prove the dimension-invariant consistency of the empirical spatial depth did not deliver a fruitful response. The authors of this paper carefully verified the proofs.
\end{Remark} 

\bibliography{literatureCOMPLETE}

@book{Wainwright2019,
  title={High-Dimensional Statistics: A Non-Asymptotic Viewpoint},
  author={Wainwright, M.J.},
  isbn={9781108498029},
  lccn={2018043475},
  series={Cambridge Series in Statistical and Probabilistic Mathematics},
  year={2019},
  publisher={Cambridge University Press}
}

@incollection{Li09,
  title={Gaussian integrals involving absolute value functions},
  author={Li, Wenbo V and Wei, Ang},
  booktitle={High dimensional probability V: the Luminy volume},
  volume={5},
  pages={43--60},
  year={2009},
  publisher={Institute of Mathematical Statistics},
  doi={10.1214/09-IMSCOLL504},
}

@article{Chaudhuri96,
author = {Probal Chaudhuri},
title = {On a Geometric Notion of Quantiles for Multivariate Data},
journal = {Journal of the American Statistical Association},
volume = {91},
number = {434},
pages = {862--872},
year = {1996},
publisher = {Taylor \& Francis},
doi = {10.1080/01621459.1996.10476954}
}

@article{Koltchinskii97,
author = {V. I. Koltchinskii},
title = {{M-estimation, convexity and quantiles}},
volume = {25},
journal = {The Annals of Statistics},
number = {2},
publisher = {Institute of Mathematical Statistics},
pages = {435 -- 477},
year = {1997},
doi = {10.1214/aos/1031833659},
}

@article{Konen26,
title={Spatial depth characterizes probability measures},
author={Gonz{\'a}lez-Sanz, Alberto and Konen, Dimitri},
journal={arXiv preprint arXiv:2607.04375},
year={2026},
doi={10.48550/arXiv.2607.04375}, 
}

@article{Chakraborty14a,
author = {Anirvan Chakraborty and Probal Chaudhuri},
title = {{The spatial distribution in infinite dimensional spaces and related quantiles and depths}},
volume = {42},
journal = {The Annals of Statistics},
number = {3},
publisher = {Institute of Mathematical Statistics},
pages = {1203 -- 1231},
year = {2014},
doi = {10.1214/14-AOS1226},
}

@Article{Chakraborty14b,
author="Chakraborty, Anirvan
and Chaudhuri, Probal",
title="On data depth in infinite dimensional spaces",
journal="Annals of the Institute of Statistical Mathematics",
year="2014",
month="Apr",
day="01",
volume="66",
number="2",
pages="303--324",
issn="1572-9052",
doi="10.1007/s10463-013-0416-y",
}

@article{VardiZhang2000,
author = {Yehuda Vardi and Cun-Hui Zhang},
title = {The multivariate $L_1$-median and associated data depth},
journal = {Proceedings of the National Academy of Sciences},
volume = {97},
number = {4},
pages = {1423-1426},
year = {2000},
doi = {10.1073/pnas.97.4.1423},
}

@InProceedings{Serfling2002,
author="Serfling, Robert",
editor="Dodge, Yadolah",
title="A Depth Function and a Scale Curve Based on Spatial Quantiles",
booktitle="Statistical Data Analysis Based on the L1-Norm and Related Methods",
year="2002",
publisher="Birkh{\"a}user Basel",
address="Basel",
pages="25--38",
isbn="978-3-0348-8201-9"
}

@article{Gao2003,
title = {Data depth based on spatial rank},
journal = {Statistics \& Probability Letters},
volume = {65},
number = {3},
pages = {217-225},
year = {2003},
issn = {0167-7152},
doi = {10.1016/j.spl.2003.06.003},
author = {Yonghong Gao},
}

@article{Konen23,
author = {Dimitri Konen and Davy Paindaveine},
title = {{Spatial quantiles on the hypersphere}},
volume = {51},
journal = {The Annals of Statistics},
number = {5},
publisher = {Institute of Mathematical Statistics},
pages = {2221 -- 2245},
year = {2023},
doi = {10.1214/23-AOS2332},
}

@article{Yeon25,
    author = {Yeon, Hyemin and Dai, Xiongtao and Lopez-Pintado, Sara},
    title = {Regularized halfspace depth for functional data},
    journal = {Journal of the Royal Statistical Society Series B: Statistical Methodology},
    volume = {87},
    number = {5},
    pages = {1553-1575},
    year = {2025},
    month = {11},
    issn = {1369-7412},
    doi = {10.1093/jrsssb/qkaf030},
}

@article{Wynne25,
author = {Wynne, George and Nagy, Stanislav},
title = {Statistical Depth Meets Machine Learning: Kernel Mean Embeddings and Depth in Functional Data Analysis},
journal = {International Statistical Review},
volume = {93},
number = {2},
pages = {317-348},
doi = {10.1111/insr.12611},
year = {2025}
}

@Article{Cuevas2007,
author="Cuevas, Antonio
and Febrero, Manuel
and Fraiman, Ricardo",
title="Robust estimation and classification for functional data via projection-based depth notions",
journal="Computational Statistics",
year="2007",
month="Sep",
day="01",
volume="22",
number="3",
pages="481--496",
issn="1613-9658",
doi="10.1007/s00180-007-0053-0",
}

@InProceedings{Tukey,

  author="Tukey, John W.",

  editor="Ralph D. James",
  
title="Mathematics and the Picturing of Data",

  booktitle="Proceedings of the International Congress of Mathematicians",
  
year="1974",
  
publisher="Canadian Mathematical Congress",
  volume="2",	
  
address="Vancouver",

  pages="523--531"

}

@article{Bartlett2002,
  title={Rademacher and gaussian complexities: Risk bounds and structural results},
  author={Bartlett, Peter L and Mendelson, Shahar},
  journal={Journal of machine learning research},
  volume={3},
  number={Nov},
  pages={463--482},
  year={2002}
}

@inproceedings{Massart2000,
  title={Some applications of concentration inequalities to statistics},
  author={Massart, Pascal},
  booktitle={Annales de la Facult{\'e} des sciences de Toulouse: Math{\'e}matiques},
  volume={9},
  number={2},
  pages={245--303},
  year={2000}
}

@inproceedings{Kakade2008,
  title={On the Complexity of Linear Prediction: Risk Bounds, Margin Bounds, and Regularization},
  author={Kakade, Sham M and Sridharan, Karthik and Tewari, Ambuj},
  booktitle={Advances in Neural Information Processing Systems},
  volume={21},
  year={2008}
}

@article{Sellke24,
title = {On size-independent sample complexity of ReLU networks},
journal = {Information Processing Letters},
volume = {186},
pages = {106482},
year = {2024},
issn = {0020-0190},
doi = {10.1016/j.ipl.2024.106482},
author = {Mark Sellke}
}

@article{Golowich20,
    author = {Golowich, Noah and Rakhlin, Alexander and Shamir, Ohad},
    title = {Size-independent sample complexity of neural networks},
    journal = {Information and Inference: A Journal of the IMA},
    volume = {9},
    number = {2},
    pages = {473-504},
    year = {2020},
    month = {06},
    issn = {2049-8772},
    doi = {10.1093/imaiai/iaz007},
}

\end{document}